\documentclass[a4paper,12pt]{article}
\usepackage{graphicx}
\usepackage{ulem}
\usepackage{tikz}
\usetikzlibrary{arrows}

\newenvironment{proof}[1][]{\noindent {\bf Proof #1:\;}}{\hfill $\Box$}

\newtheorem{proposition}{Proposition}
\newtheorem{corollary}{Corollary}

\newtheorem{definition}{Definition}
\newtheorem{assumption}{Assumption}

\newcommand{\RR}{\mathbb{R}}

\newcommand{\jb}{\color{red}}
\renewcommand{\jb}{}
\newcommand{\valFun}{v}
\newcommand{\traj}{\rm traj}

\usepackage{amsmath}
\usepackage{amssymb}
\usepackage{amsfonts}
\usepackage{graphicx}
\usepackage{mathrsfs}

\begin{document}

\author{Edouard Pauwels$^{1}$, Didier Henrion$^{2,3}$, Jean-Bernard Lasserre$^{2}$}
\title{\bf Positivity certificates in optimal control}
\footnotetext[1]{IRIT, Universit\'e de Toulouse, Toulouse, France.}
\footnotetext[2]{LAAS-CNRS, Universit\'{e} de Toulouse, CNRS, Toulouse, France.}
\footnotetext[3]{Faculty of Electrical Engineering, Czech Technical University in Prague, Prague, Czech Republic.}
\date{Draft of \today}

\maketitle

\begin{abstract}
	We propose a tutorial on relaxations and weak formulations of optimal control with their semidefinite approximations. We present this approach solely through the prism of positivity certificates which we consider to be the most accessible for a broad audience, in particular in the engineering and robotics communities. This simple concept allows to express very concisely powerful approximation certificates in control. The relevance of this technique is illustrated on three applications: region of attraction approximation, direct optimal control and inverse optimal control, for which it constitutes a common denominator. In a first step, we highlight the core mechanisms underpinning the application of positivity in control and how they appear in the different control applications. This relies on simple mathematical concepts and gives a unified treatment of the applications considered. This presentation is based on the combination and simplification of published materials. In a second step, we describe briefly relations with broader literature, in particular, occupation measures and Hamilton-Jacobi-Bellman equation which are important elements of the global picture. We describe the Sum-Of-Squares (SOS) semidefinite hierarchy in the semialgebraic case and briefly mention its convergence properties. Numerical experiments on a classical example in robotics, namely the nonholonomic vehicle, illustrate the concepts presented in the text for the three applications considered.

\end{abstract}


\section{Introduction}
\subsection{Context}
In the context of understanding and reproducing human movements and more generally in motion control, there has recently been a growing interest in using optimal control to model and account for the complexity of underlying processes \cite{arechavaleta2008optimality,chittaro2013inverse,friston2011what,laumond2014optimality,majumdar2014convex,mombaur2010human,puydupin2012convex,todorov2004optimality}. The question of the validity of this approach is still open and the interface between optimal control and human locomotion {\jb is an active field of research}.

The {\jb so-called  weak formulation} of optimal control problem {\jb has} a long history in the control community \cite{vinter1978equivalence,hernandez1996linear,friston2011what}, see also \cite[Part III]{fattorini1999inifinite} for a detailed historical perspective. This approach comes with a rich convex duality structure \cite{vinter1993convex}, one side of which involves functional non negativity constraints. This type of constraints constitutes the focus of this paper.

{\jb In general,} functional positivity constraints are not tractable computationally. Advances {\jb in semialgebraic geometry 
on the representation of positive polynomials 
\cite{putinar1993positive}} have allowed to construct provably {\jb {\it convergent hierarchies of sums-of-squares approximations}} to this kind of intractable constraints when the problem at hand only involves polynomials \cite{lasserre2001global,lasserre2010}. 
{\jb Based on semidefinite programming \cite{vandenberghe1996semidefinite}, these approximations provide
a new perspective on infinite dimensional linear programs and functional non negativity constraints, along with 
tractable numerical approximations}. Application of these hierarchies {\jb in control} lead to the design {\jb of new methods to address} control problems with global-optimality guaranties \cite{lasserre2008nonlinear,parrilo2003semidefinite,henrion2004solving,henrion2005positive,chesi2010lmi}.
\subsection{Content}
This paper is a tutorial which focuses on {\jb the application} of infinite dimensional conic programming approach to control problems. This constitutes a very relevant tool for the human locomotion and humanoid robotics research communities. Indeed the {\jb sums-of-squares (SOS) hierarchy} 
{\jb provides a systematic numerical scheme to solve} related practical control problems. {\jb We will illustrate the power of this approach by focusing on three such particular problems, namely:
	\begin{itemize}
		\item region of attraction approximation.
		\item direct optimal control.
		\item inverse optimal control.
	\end{itemize}
}

The infinite dimensional linear programming approach {\jb combined with its associated SOS hierarchy of approximations} has been applied to these problems in \cite{lasserre2008nonlinear,henrion2014convex,pauwels2014linear,pauwels2014inverse}. All in all, the contents of this paper is not new {\jb and is merely based} on existing materials from the control and {\jb sums-of-squares approximations} literature. The purpose of the paper is to {\jb reveal and highlight a few simple mechanisms and ideas that constitute a common denominator of all these applications.}

{\jb Being concerned with accessibility to a broad audience, we deliberately hide important aspects of the approach.} In particular, we focus on functional positivity constraints {\jb (one side of a coin in this approach)} because we think that this is {\jb the most accessible way to present a general intuition regarding the weak formulation of optimal control problems}. {\jb Another reason is that this simple notion of positivity allows to provide very strong sub-optimality certificate stemming from elementary mathematics}. Other facets of the same problem {\jb (the other side of the coin described in the dual of the infinite-dimensional linear program), including conic duality and details about the weak formulation of control problems
on {\it occupation measures}, are only briefly mentioned in a second step with very few details}. Indeed, this material is often {\jb perceived} as more technical and less accessible from a mathematical point of view. {\jb Although we do not emphasize much the moment relaxation approximation and its relation with occupation measures, it would provide a more complete picture to speak about the {\it moment-SOS hierarchy} (instead of the SOS-hierarchy) because each semidefinite program of the 
SOS-hierarchy of approximations of functional positivity constraints has a dual (also a semidefinite program) which deals with ``moments" of
occupation measures. We mention this point only briefly and invite the reader interested in more details about these aspects to consult the existing literature}.

\subsection{Organization of the paper}
The optimal control problem and its value function are introduced in Section \ref{sec:optimalControl}. In Section \ref{sec:bounds}, we introduce functional positivity constraints which involve surrogate value functions. We discuss implications of these types of constraints in the context of optimal control, {\jb and in particular we describe how they relate to the approximation of the value function}.  This constitutes a general and flexible {\jb core} result that is useful in the control applications that we consider. Section \ref{sec:applications} {\jb 
illustrates the concept in several control problems dealing with (i) the approximation of region of attraction, (ii) optimal control and (iii), inverse optimal control.} {\jb Finally,} Section \ref{sec:convergence} discusses connections with {\jb the optimal control literature} and additional aspects of the approach that we do not {\jb describe explicitly}. We also briefly describe how the {\jb sums-of-squares} hierarchy of approximations can be {\jb implemented} and discuss convergence issues.
\section{Optimal control and value function} 
\label{sec:optimalControl}
\subsection{Notations and preliminaries}
If $A$ is a subset of $\RR^n$, $\mathcal{C}(A)$ {\jb denotes the space} of continuous functions from $A$ to $\RR$ 
{\jb while $\mathcal{C}^1(A)$ denotes the space of continuously differentiable functions from $A$ to $\RR$}. 
Let $X \subseteq \RR^{d_X}$ {\jb 
and $U \subseteq \RR^{d_U}$ denote respectively the state and control spaces, both supposed to be compact}. {\jb The system dynamics is} given by a continuously differentiable vector field $f \in \mathcal{C}^1(X \times U)^{d_X}$. Terminal state constraints are represented by a {\jb given compact set $X_T \subseteq X$.}

{\jb Given all the above ingredients one may define {\it admissible trajectories}} in the context of optimal control. We will use the following definition.
\begin{definition}[Admissible trajectories]\label{def:admissibleTraj}
	Consider an initial time $t_0 \in [0,1]$ and a pair of functions $(x, u)$ from $[t_0,1]$ to $\RR^{d_X}$ and $\RR^{d_U}$ respectively. This pair constitutes an admissible trajectory if it has the following properties:
	\begin{itemize}
		\item $u$ is a measurable function from $[t_0,1]$ to $U$. 
		\item For any $t \in [t_0,1]$, $x(t) = x_0+\int_{t_0}^t f(x(s),u(s))ds$.
		\item $x(1) \in X_T$.
	\end{itemize}
	Given $x_0 \in X$, {\jb denote} by $\traj_{t_0, x_0}$ the set of such admissible trajectories starting at time $t_0$ with $x(t_0) = x_0$. Note that the second property implies that $x$ is differentiable almost everywhere as a function of $t$, with $\dot{x}(t) = f(x(t), u(t))$ for almost all $t \in [t_0, 1]$.
\end{definition}
The class of admissible trajectories constitutes the decision variables of an optimal control problem.
\subsection{Optimal control and value function}
An optimal control problem consists of minimizing a functional over the set of admissible trajectories. The functional has a specific integral form involving a continuous {\jb Lagrangian $l \in \mathcal{C}(X \times U)$ and a continuous terminal cost $l_T \in \mathcal{C}(X_T)$}. Given an initial time $t_0 \in [0,1]$ and a starting point $x_0 \in X$, {\jb consider the infimum value:
\begin{equation}\label{eq:optimalControl}
	\tag{OCP}
	\begin{array}{llll}
		v^*(t_0,x_0) & := & \inf & \displaystyle \int_{t_0}^1 l(x(t),u(t))dt + l_T(x(1)) \\
		&&\mathrm{s.t.} & {(x,u) \in \traj_{t_0, x_0}}
	\end{array}
\end{equation}
of the functional over all admissible trajectories. It is a well defined value that only depends on $t_0$ and $x_0$
and $v^*:[0,1]\times X\to \mathbb{R}\cup\{+\infty\}$  is called the {\it value function} associated with the optimal control problem.}

Note that the constraints in (\ref{eq:optimalControl}) ensure that we {\jb only consider} admissible trajectories starting from $x_0$ at $t_0$,
{\jb and therefore if $\traj_{t_0, x_0}$ is empty then $v^*(t_0, x_0) = +\infty$}.
\section{Bounds on the value function}
\label{sec:bounds}
The value function introduced in (\ref{eq:optimalControl}) can be a very complicated object. The existence of minimizing sequences, the question of the infimum being attained and the regularity of $v^*$ {\jb are all quite delicate issues}. In this section we show that functional positivity constraints that are expressible in {\jb a simple form} lead to powerful approximation results. 
{\jb In addition, and remarkably, a striking feature of these results is that their proof arguments are elementary}. We {\jb now focus} on the description of these constraints {\jb while} their origin and connection with control theory are {\jb postponed to} Section \ref{sec:convergence}.
\subsection{Global lower bounds}
We let {\jb ``$\cdot$"} denote the dot product between two vectors of the same size. For a given function $\valFun \in \mathcal{C}^1( [0,1] \times X)$, {\jb consider} the following positivity conditions:
\begin{align}\label{eq:globalLower}
	l(x,u) + \frac{\partial \valFun}{\partial t}(x,t) + \frac{\partial \valFun}{\partial x}(x,t) \cdot f(x,u) \geq 0 &\quad\; \forall (x, u, t) \in X \times U \times [0,1]\\
	l_T(x) - \valFun(T,x) \geq 0 &\quad\; \forall x\in X_T.\nonumber
\end{align}
Note that these conditions are indeed functional positivity constraints since both of them must hold for all arguments in certain sets. 
{\jb How to ensure or approximate such conditions in practical situations is discussed in Section \ref{sec:sosReinforcement}}. We focus for the moment on the consequences of condition (\ref{eq:globalLower}) in terms of control, the following proposition being an elementary, yet powerful example.
\begin{proposition}[Global lower bound on the value function]\label{prop:globalLower}
{\jb If $\valFun \in \mathcal{C}^1([0, 1] \times X)$ satisfies condition (\ref{eq:globalLower}) then $\valFun(t_0, x_0) \leq v^*(t_0,x_0)$ for any $x_0 \in X$ and $t_0 \in [0,1]$.}
\end{proposition}
\begin{proof}
	Fix $x_0 \in X$ and $t_0$ and consider the set $\traj_{t_0, x_0}$ of admissible trajectories starting at $x_0$ at time $t_0$ as described in Definition \ref{def:admissibleTraj}. If this set is empty {\jb then $v^*(t_0,x_0) = +\infty$}. Since $\valFun$ is continuous on a compact set, it is bounded and hence finite at $(t_0, x_0)$ which ensures that $\valFun(x_0,t_0) \leq v^*(x_0,t_0)$. If $\traj_{t_0, x_0}$ is {\jb not empty, consider} an arbitrary but fixed admissible trajectory $(x, u)\colon [t_0, 1] \to X \times U$ which satisfies all the requirements of Definition \ref{def:admissibleTraj} with $x(t_0) = x_0$. Combining admissibility with the first condition in (\ref{eq:globalLower}) {\jb yields:}
	
	\begin{align*}
		l(x(t),u(t)) + \frac{\partial}{\partial t}\left[ \valFun(t, x(t)) \right]
		&=l(x(t),u(t)) + \frac{\partial \valFun}{\partial t}(t, x(t)) + \frac{\partial \valFun}{\partial x}(t, x(t)) \cdot \dot{x}(t)\\
		&=l(x(t),u(t)) + \frac{\partial \valFun}{\partial t}(t, x(t)) + \frac{\partial \valFun}{\partial x}(t, x(t)) \cdot f(x(t),u(t))\\ 
		&\geq 0,\quad{\jb \text{ for almost all } t\in [t_0,1].}
	\end{align*}
	{\jb Integrating between $t_0$ and $1$, and using} non negativity of the first term, we obtain
	\begin{align*}
		{\jb \int_{t_0}^1 l(x(t),u(t))\,dt} + \valFun(1, x(1)) - \valFun(t_0, x_0) \geq 0.
	\end{align*}
	{\jb Combining with the second condition in (\ref{eq:globalLower}) yields}
	\begin{align*}
		\valFun(t_0, x_0) \leq {\jb \int_{t_0}^1 l(x(t),u(t))\,dt} + l_T(x(1)).
	\end{align*}
	Since $(x, t)$ was arbitrary among all admissible trajectories, {\jb this inequality is still valid
	if we take the infimum in the right hand side, which coincides with the definition of $v^*$ in (\ref{eq:optimalControl}), and 
	the proof is complete.}
\end{proof}

Proposition \ref{prop:globalLower} provides a sufficient condition to obtain global lower bounds on the value function $v^*$. A remarkable property of this condition is that it {\jb {\it does not depend explicitly on $v^*$}}. In particular, condition (\ref{eq:globalLower}) does not {\jb depend explicitly} on regularity properties of $v^*$ or {\jb on} the existence of {\jb optimal} trajectories in (\ref{eq:optimalControl}). Furthermore, they are expressed in a relatively compact form and the proof arguments are elementary.

\subsection{Local upper bounds}
We now turn to upper bounds on the value function $v^*$ of problem (\ref{eq:optimalControl}). First, {\jb observe}
that if the set of admissible trajectories is empty in (\ref{eq:optimalControl}) {\jb then $v^*(t_0,x_0)=+\infty$}. 
{\jb Hence upper bounding $v^*$ using a continuous function only makes sense} when the set of admissible trajectories is not empty. Therefore such upper bounds {\jb depend} on admissible trajectories and only hold in a certain {\jb ``local sense"}. 
In particular{\jb ,  global upper bounds do not exist in general, whence the local characteristic for the type of bounds derived in this section}. We introduce the following notation 
\begin{definition}[Domain of the value function]\label{def:suppValFun}
{\jb Denote by $V \subset [0,1] \times X$ the domain of $v^*$, that is,} the subset of $[0,1]\times X$ on which $v^*$ takes finite values,
\begin{align*}
	V := \left\lbrace(t_0,x_0) \in [0,1] \times X: \:\traj_{t_0,x_0} \neq \emptyset \right\rbrace
\end{align*}
\end{definition}
{\jb Consider} a fixed pair $(t_0, x_0) \in V$ and a fixed admissible trajectory {\jb $(x,u) \in \traj_{t_0,x_0}$, starting at $x_0$ at time $t_0$.} 
For a given $\epsilon \geq 0$, the following conditions are a counterpart to the positivity condition in (\ref{eq:globalLower}).
\begin{align}\label{eq:localUpper}
	l(x(t),u(t)) + \frac{\partial \valFun}{\partial t}(t, x(t)) + \frac{\partial \valFun}{\partial x}(t, x(t)) \cdot f(x(t),u(t)) &\leq \frac{\epsilon}{2}, \; \text{ for almost all } t \in [t_0, 1]\\ 
	l_T(x(T)) - \valFun(1,x(1)) &\leq \frac{\epsilon}{2}\nonumber.
\end{align}

{\jb They can be used to obtain the following upper approximation result}.
\begin{proposition}[Local upper bound on the value function]\label{prop:localUpper}
	Let $(t_0, x_0) \in V$ be fixed. Let $(x,u) \in \traj_{t_0,x_0}$ be an admissible trajectory starting at $x_0$ at time $t_0$. Assume that $\valFun \in \mathcal{C}^1([0, 1] \times X)$ satisfies condition (\ref{eq:localUpper}) for a given $\epsilon >0$. {\jb Then 
	$v^*(t,x(t)) \leq \valFun(t, x(t)) + \epsilon$ for all $t \in [t_0, 1]$. In addition, if $\valFun$ satisfies condition (\ref{eq:globalLower}) then $(x,u)$ is at most $\epsilon$ sub-optimal for problem (\ref{eq:optimalControl}): feasible with objective value at most $\epsilon$ greater than the optimal value.}
\end{proposition}
\begin{proof}
	Following similar integration arguments as in the proof of Proposition \ref{prop:globalLower}, using the first part of condition (\ref{eq:localUpper}) {\jb yields:}
	\begin{align*}
		{\jb \int_{t}^1 l(x(s),u(s))\,ds} + \valFun(1, x(1)) - \valFun(t, x(t)) \leq (1 - t) \frac{\epsilon}{2} \leq \frac{\epsilon}{2},\quad{\jb \forall\,t\in [t_0, 1],}
	\end{align*}
	{\jb and combining with the second part of condition (\ref{eq:localUpper}),} 
	\begin{align*}
		{\jb \int_{t}^1 l(x(s),u(s))\,ds}+  l_T(x(1)) \leq l_T(x(1)) - \valFun(1, x(1)) + \valFun(t, x(t)) + \frac{\epsilon}{2}  \leq \valFun(t, x(t)) + \epsilon,
	\end{align*}
{\jb for all $t\in [t_0, 1]$.}
{\jb As the} left hand side is an upper bound on $v^*(t, x(t))$, {\jb the first statement follows}. {\jb In addition, if} condition (\ref{eq:globalLower}) is satisfied {\jb then we can use Proposition \ref{prop:globalLower} at $(t, x(t))$ to obtain:}
	\begin{align*}
		{\jb \int_{t}^1 l(x(s),u(s))\,ds} +  l_T(x(1)) \leq v^*(t, x(t)) + \epsilon.
	\end{align*}
	In particular, {\jb letting $t=t_0$ in the previous relation yields that $(x,u)$ is at most $\epsilon$-sub-optimal for problem (\ref{eq:optimalControl}).}
\end{proof}

{\jb Again,} a remarkable property of condition (\ref{eq:localUpper}) is that it depends neither on the regularity of $v^*$ nor on the existence of optimal trajectories and still provides powerful {\jb sub-optimality} certificates. Note that {\jb 
Proposition \ref{prop:localUpper} characterizes properties of $v^*$ only along} the specific chosen trajectory, {\jb whence} the name ``local'' for {\jb this type} of bounds.

\section{Applications in control}
\label{sec:applications}
In this section, we consider applications in control and show how conditions (\ref{eq:globalLower}) and (\ref{eq:localUpper}) can be used to solve control  problems. 

\begin{center}
{\jb {\it The general methodology is to use conditions (\ref{eq:globalLower}) and (\ref{eq:localUpper}) as constraints in combination with additional constraints and linear objective functions depending on the application.}}
\end{center} 

The reason why this is relevant and produces valid practical methods comes from the connection with Propositions \ref{prop:globalLower} and \ref{prop:localUpper}. Depending on the problem at hand, definition of objective functions or addition of constraints {\jb allow to 
provide a systematic numerical scheme to solve the control problems we consider: approximating the region of attraction of a controlled system, solving optimal control and inverse optimal control problems, provided that they are described with polynomials and semi-algebraic sets} (see also Section \ref{sec:sosReinforcement}). All the material of this section is based on reformulation and simplification of the work presented in \cite{lasserre2008nonlinear,henrion2014convex,pauwels2014inverse,pauwels2014linear}.

\subsection{Region of attraction}
The region of attraction is a subset of the domain of the value function, $V$ in Definition \ref{def:suppValFun}, corresponding to a fixed initial time $t_0$. In other words, we are looking for the set $X_0$ of initial conditions, {\jb $x_0$}, for which there exists an admissible trajectory starting {\jb in state $x_0$} at a given time $t_0$.
\begin{definition}[Region of attraction]
	The region of attraction at time $t_0$, denoted by $X_0 \subset{X}$, is the set that satisfies
	\begin{align*}
		X_0 = \left\{ x_0 \in X: \traj_{t_0,x_0} \neq \emptyset \right\},
	\end{align*}
	where $\traj_{t_0, x_0}$ is the set of admissible trajectories as given in Definition \ref{def:admissibleTraj}. Following Definition \ref{def:suppValFun}, we have $\left\{ t_0 \right\} \times X_0 = V \cap \left\{ t_0 \right\} \times X$.
\end{definition}

This exactly corresponds to the situation where $l = 0$ and $l_T=0$ in (\ref{eq:optimalControl}). Indeed, in this case, $v^*$ becomes the indicator of $X_0$ (equal to $0$ on $X_0$ and $+\infty$ otherwise) and the optimal control problem is a feasibility problem.

Condition (\ref{eq:globalLower}) becomes
\begin{align}\label{eq:ROA}	
	\frac{\partial \valFun}{\partial t}(x,t) + \frac{\partial \valFun}{\partial x}(x,t) \cdot f(x,u) \geq 0 &\quad\; \forall (x, u, t) \in X \times U \times [0,1]\\
	\valFun(T,x) \leq 0 &\quad\; \forall x\in X_T\nonumber
\end{align}
and Proposition \ref{prop:globalLower} has the following consequence.
\begin{corollary}\label{cor:ROA}
	{\jb If  $\valFun \in \mathcal{C}^1([0, 1] \times X)$ satisfies condition (\ref{eq:ROA}) then $\valFun(x_0,t_0) \leq 0$ for any $x_0 \in X_0$.}
\end{corollary}
Corollary \ref{cor:ROA} {\jb states} that $X_0$ is contained in the zero sublevel set of $\valFun$ whenever $\valFun$ satisfies condition (\ref{eq:ROA}). {\jb However this is} not sufficient to have a good approximation of $X_0$ and condition (\ref{eq:ROA}) is not strong enough to distinguish between accurate and rough sublevel set approximations of this type. In order to sort out accurate candidates $\valFun$, a possibility is to search among all functions which satisfy condition (\ref{eq:ROA}) {\jb an ``optimal" one, e.g., in the sense that it} should be as much as possible greater than 0 outside of $X_0$. Following \cite{henrion2014convex}, we introduce an additional decision variable $w \in \mathcal{C}(X)$. We will construct an optimization problem which ensures that $w$ is non positive on $X_0$ and as close as possible to $1$ on $X \setminus X_0$. This can be obtained by combining Corollary \ref{cor:ROA} with additional positivity constraints and a linear objective function. The following problem is a reformulation of problem (16) in \cite{henrion2014convex}.
\begin{equation}
	\label{eq:ROAopt}
	\begin{array}{ll}
		\displaystyle\sup_{\valFun,w} &\int_X w(x) dx \\
		\mathrm{s.t.}&0\,\leq\,\frac{\partial \valFun}{\partial t} + \frac{\partial \valFun}{\partial x} \cdot f\\
		&0 \,\leq \,- \valFun(T,\cdot)\\
		& w(\cdot) \leq v(\cdot, t_0)\\
		& w \leq 1.
	\end{array}
\end{equation}
In problem (\ref{eq:ROAopt}), the first to constraints are exactly condition (\ref{eq:ROA}) and corollary \ref{cor:ROA} ensures that $\valFun(\cdot,t_0) \leq 0$ on $X_0$. Therefore, the third constraint ensures that $w \leq 0$ on $X_0$. The last constraint combined with the objective function allow to ``choose'' $w$ as close as possible to $1$ on $X \setminus X_0$. In general the supremum in (\ref{eq:ROAopt}) is not attained, but any candidate solution $w$, is such that its zero sublevel contains $X_0$ and remains close to it in a certain sense. Indeed it what shown in \cite{henrion2014convex} that the supremum in (\ref{eq:ROAopt}) is equal to the volume of $X_0$ and this quantity can be approximated by hierarchies of semidefinite approximations which we describe in Section \ref{sec:convergence}.


\subsection{Optimal control}
In this section, we fix $t_0$ and $x_0$. As described in Section \ref{sec:bounds}, {\jb condition (\ref{eq:globalLower}) provides }a global lower bound on $v^*$. {\jb However, the} family of functions $\valFun$ which satisfy this condition is too {\jb large}. For example, if $l \geq 0$ and $l_T \geq 0$, then $\valFun = 0$ satisfies condition (\ref{eq:globalLower}) and does not provide much insight {\jb regarding} solutions of (\ref{eq:optimalControl}). Therefore, {\jb one} should design a way to choose lower bounds of specific interest. {\jb In} the {\jb (direct)} optimal control problem, {\jb one is} interested in the value $v^*(t_0, x_0)$. Hence an informal approach is to choose among all $\valFun$ that satisfy condition (\ref{eq:globalLower}) one for which $\valFun(t_0, x_0)$ is close to $v^*(t_0, x_0)$. Note that under condition (\ref{eq:globalLower}) we {\jb already} have $\valFun(t_0, x_0) \leq v^*(t_0, x_0)$ and hence it is sufficient to look for 
{\jb a function $\valFun$ such that $\valFun(t_0, x_0)$ is as large as possible}. This leads to the following optimization problem.
\begin{equation}
	\label{eq:ROCP}
	\begin{array}{ll}
		{\jb \displaystyle\sup_\valFun} & \valFun(t_0, x_0)\\
		\mathrm{s.t.}&0\,\leq\, l + \frac{\partial \valFun}{\partial t} + \frac{\partial \valFun}{\partial x} \cdot f\\
		&0 \,\leq \,l_T(\cdot) - \valFun(T,\cdot).
	\end{array}
\end{equation}
{\jb In general the supremum is not attained}. Furthermore, for most reasonable practical situations, the value of the problem is exactly $v^*(t_0, x_0)$, providing a valid conceptual solution to the optimal control problem.

{\jb At this point a remark is in order}. Solutions of problem (\ref{eq:ROCP}) allow to  approximate from below the value function $v^*$. In this {\jb respect} they provide solutions of (\ref{eq:optimalControl}) because of their relations {\jb to} $v^*$ which is the value of specific interest. However, this approach does not give access to an optimal trajectory which achieves this optimal value. Indeed, without further assumptions, the existence of such an optimal trajectory is not guaranteed. In order to compute optimal trajectories, further conditions are required in combination with additional methods to search for optimal trajectories. When such a method is available, it is always possible to combine it with solutions of (\ref{eq:ROCP}) by using condition (\ref{eq:localUpper}) and Proposition \ref{prop:localUpper} to certify the {\jb sub-optimality} of the computed trajectory.

\subsection{Inverse optimal control}

{\jb In inverse optimal control the situation} is somewhat reversed compared to direct optimal control. The Lagrangian is unknown
{\jb but we are given a set of trajectories that should be optimal with respect to the unknown Lagrangian}. {\jb So the goal} is to find a Lagrangian for which the given trajectories {\jb are} optimal. {\jb In Figure \ref{fig:directIllustr}, we display an informal description of this problem and its relation with the direct optimal control problem in the framework of positivity {\jb certificates}}. Briefly, the main goal is to infer a cost function {\jb (a Lagrangian) which can generate} a set of given {\jb trajectories} through an optimality process. The applications of this are twofold:
\begin{itemize}
	\item Provide {\jb a tool} for applications {\jb in which one assumes the existence of an optimality process behind decisions}.
	\item Provide {\jb a modeling tool} which could allow to summarize and reproduce {\jb the} behaviour of observed systems. 
\end{itemize}
\begin{figure}
	\begin{center}
		\begin{tikzpicture}[
		  every matrix/.style={ampersand replacement=\&,column sep=1cm,row sep=.1cm},
		  oneBox/.style={draw,thick,rounded corners,inner sep=.1cm,align=left},
		  twoBox/.style={align=left},
		  to/.style={->,>=stealth',shorten >=1pt,semithick,font=\sffamily\footnotesize}]
		
		  \matrix{
				\node[twoBox] (box01) {\underline{\textbf{System description}: $f, X, U, X_T$}};\&\\
				\&\node[twoBox] (box01) {\flushleft{\textbf{Direct control}:}};\\
				\node[oneBox] (box1) {\textbf{Input:}\\Lagrangian $l$};
				\& \node[oneBox] (box2) {\textbf{Positivity:}\\
																			\texttt{Data:} $l$\\
																			\texttt{Unknown:} $x(t), \valFun$
		
							};
				\& \node[oneBox] (box3) {\textbf{Output:}\\Optimal\\trajectory: $x(t)$};\\
				\&\\
				\&\\
				\&\node[twoBox] (box01) {\textbf{Inverse control}:};\\
				\node[oneBox] (box21) {\textbf{Input:}\\Controlled\\trajectories: $x(t)$};
				\& \node[oneBox] (box22) {\textbf{Positivity:}\\
																			\texttt{Data:} $x(t)$\\
																			\texttt{Unknown:} $l, \valFun$
		
							};
					\& \node[oneBox] (box23) {\textbf{Output:}\\Lagrangian $l$};\\
		  };
		
		  \draw[to] (box1) -- (box2);
		  \draw[to] (box2) -- (box3);
		  \draw[to] (box21) -- (box22);
		  \draw[to] (box22) -- (box23);
		\end{tikzpicture}
	\end{center}
	\caption{\label{fig:directIllustr} Direct optimal and inverse optimal control flow chart. {\jb The dynamical system is described 
	through the dynamics $f$, the state constraint set $X$, control constraint set $U$ and terminal state constraint set $X_T$ which are all given}. We emphasize that the Lagrangian and the trajectories have symmetric roles for the direct and inverse problems. {\jb In particular,}  the output of the inverse problem is a Lagrangian.}
\end{figure}
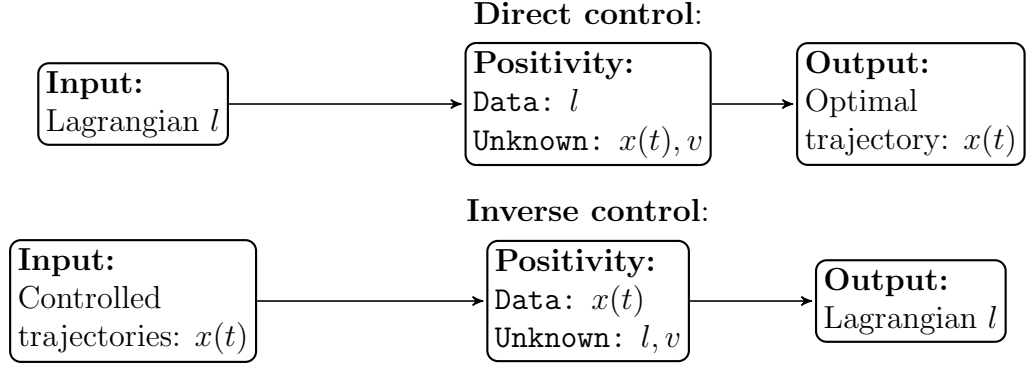

{\jb In the} rest of this section, we fix an admissible trajectory $(x,u)$ starting from $x_0 \in X_0$ at time $0$. We suppose that the state trajectory $x$ as well as the control {\jb trajectory} $u$ are given and we look for candidate Lagrangians. The whole methodology {\jb naturally extends to an arbitrary number of trajectories. Actually, the higher the number of trajectories, the better and the more (physically) meaningful is the characterization
of the candidate Lagrangian that we are looking for}. {\jb However for clarity of exposition, the approach is
better understood when we consider a single given trajectory}.

In order to provide a solution to the inverse problem, we combine conditions (\ref{eq:globalLower}) and (\ref{eq:localUpper}). The relevance of doing this comes from Proposition \ref{prop:localUpper} which provides a {\jb sub-optimality} certificate. In addition, we enforce $l_T= 0$ in order to simplify the problem and avoid the multiplication of non identifiabilities. Among all potential certificates we look for the one that provides the smallest sub-optimality gap as described in Proposition \ref{prop:localUpper}. This leads to the following optimization problem.

\begin{equation}
	\label{eq:IOCP}
	\begin{array}{llll}
		{\jb \displaystyle\inf_{\epsilon, l,\valFun}} & \epsilon&&\\
		\mathrm{s.t.}&0&\leq& l + \frac{\partial \valFun}{\partial t} + \frac{\partial \valFun}{\partial x} \cdot f\\
		&0 &\leq& l_T(\cdot) - \valFun(T,\cdot)\\
		& \frac{\epsilon}{2} &\geq&l(x(t),u(t)) + \frac{\partial \valFun}{\partial t}(t, x(t)) + \frac{\partial \valFun}{\partial x}(t, x(t)) \cdot f(x(t),u(t))\; \forall t \in [0, 1]\\ 
		&\frac{\epsilon}{2}&\geq&l_T(x(T)) - \valFun(1,x(1)). 
	\end{array}
\end{equation}
{\jb By Proposition \ref{prop:localUpper}, if $l$ is a Lagrangian part of a feasible solution $(\epsilon,l,\valFun)$
for problem (\ref{eq:IOCP}), then the trajectory $(x,u)$ is $\epsilon$-sub-optimal
for problem (\ref{eq:optimalControl}) with Lagrangian $l$. In other words, every feasible solution $(\epsilon,l,\valFun)$ of (\ref{eq:IOCP}) provides us with
an $\epsilon$-sub-optimality certificate for the trajectory $(x, u)$.}

{\jb However, this is not sufficient}. Indeed, problem (\ref{eq:IOCP}) {\jb always admits the trivial solution $(0,0,0)$ and it turns out that this solution is also valid from a formal point of view.} Indeed, {\jb every} admissible trajectory is optimal for the trivial {\jb null} Lagrangian, and therefore, from the point of view of inverse optimality, {\jb the null Lagrangian is a valid (but not satisfactory) solution. To avoid trivial Lagrangians, 
additional constraints on $l$ are needed}. We will settle upon problem (\ref{eq:IOCP}) as it highlights the main mechanism behind positivity in inverse optimal control and invite the reader to see \cite{pauwels2014inverse,pauwels2014linear} for further discussions and more details about application in practical situations.

\section{Duality, Hamilton-Jacobi-Bellman, SOS reinforcement and convergence}
\label{sec:convergence}
The results presented so far without much context elements are related to principles which have a long history in optimal control theory. 
{\jb In this section we mention a few of them and we also comment on how to use these results in practical contexts through the SOS hierarchy}.

\subsection{Occupation measures}
\label{sec:conicDuality}
{\jb The constraints imposed in (\ref{eq:globalLower}) have a {\it conic flavor} as they combine} linear operators and positivity constraints. The {\jb space of continuous functions that are nonnegative on a given set form a convex cone}. This cone admits a {\jb (convex) dual cone}, see \cite[Chapter IV]{alexander2002course} for a description of conic duality in Banach spaces. Representation results of Riesz type ensure that this dual cone can be identified with that of nonnegative measures {\jb on} the same set. {\jb As is classical for duality in convex optimization:
	\begin{itemize}
		\item To the inequality constraints appearing in the conic optimization problem (\ref{eq:globalLower}) are associated nonnegative {\it dual variables} in the dual conic optimization problem, and
		\item to the variables appearing in (\ref{eq:globalLower}) are associated constraints on these dual variables.
	\end{itemize}
The constraints in the dual problem describe a transport equation satisfied by the dual variables, more precisely the transport along the flow
followed by admissible trajectories in Definition \ref{def:admissibleTraj}. These dual variables are called {\jb ``occupation measures",}
see e.g. \cite{vinter1993convex} for an accurate description.} 

{\jb In other words,} the dual counterpart of condition (\ref{eq:globalLower}) allows to {\jb formally} work with {\jb {\it generalized trajectories}} instead of classical ones. {\jb Whence} the name ``relaxation'' for this approach. {\jb Equivalently,} one speaks of ``weak formulation'' of the optimal control problem (\ref{eq:optimalControl}) because the differential equation is replaced by a weaker constraint
{\jb (the transport equation for occupation measures)}. {\jb One} main benefit of working with weak formulations is that the question of attaining the infimum {\jb is solved, at least from a theoretical point of view, under weak conditions, e.g. compactness of the sets $X$ and $X_T$). However, the relaxed problem is not equivalent to the original problem, and its optimal value may be smaller. But for} most reasonable practical situations, there is no relaxation gap and {\jb the optimal values} of both problems are the same \cite{vinter1993convex,vinter1978equivalence}. Although the use of occupation measures is much less popular than classical differential equations in the engineering community, it is classical in Markov processes and ergodic dynamical systems. Furthermore, understanding this dual aspects is crucial in the framework of positivity constraints that we describe. 

\subsection{Hamilton-Jacobi-Bellman equation}
Conditions (\ref{eq:globalLower}) and (\ref{eq:localUpper}) have the same {\jb flavor and structure} as the well known Hamilton-Jacobi-Bellman (HJB) sufficient optimality conditions (see e.g. exposition in \cite{athans1966optimal}). In fact, if we combine (\ref{eq:globalLower}) with (\ref{eq:localUpper}) with $\epsilon = 0$, we recover exactly the same condition. This {\jb provides} a certificate of optimality {\jb for} a given trajectory. However, this condition is {\jb not} necessary. {\jb Indeed if the value function $v^*$ is not smooth 
(which is the case in most practical situations) then 
it is not possible to fulfill this condition in the classical sense.}
{\jb Whence} the use of a relaxed condition involving {\jb $\epsilon>0$} that measures how far we are from the true optimality condition. Another possible workaround is the use of the elegant {\jb {\it viscosity solution}} concept to define {\jb ``solutions"} of HJB equation \cite{bardi2008optimal}. This involves a lot more {\jb sophisticated} mathematical machinery, far beyond the scope of this paper.

\subsection{SOS reinforcement}
\label{sec:sosReinforcement}
Finally, conditions (\ref{eq:globalLower}) and (\ref{eq:localUpper}) are actually positivity constraints for functions. Moreover, all the examples presented in Section \ref{sec:applications} consist of combining these constraints with additional constraints of the same type and linear objective functions. {\jb In full generality this} type of constraints is not amenable to practical computation. In order to be able to use the results of this paper to actually solve control problems, involving 
{\jb some practical ``algorithm"}, we need to enforce more structure on the objects we manipulate. A {\jb now} widespread approach is to work with the following assumption.
\begin{assumption}\label{ass:algebra}
	The dynamics $f$, the Lagrangian $l$ and the terminal cost $l_T$ are polynomials. Constraints set $X$, $U$ and $X_T$ are {\jb compact basic semi-algebraic sets}.
\end{assumption}
Recall that a {\jb closed basic semi-algebraic} set $G$ can is defined by inequalities involving a finite number of polynomials $g_1, \ldots, g_q \in \RR\left[ X \right]$:
\begin{align}\label{eq:setG}
	G = \left\{ x:\; g_i(x) \geq 0,\,i=1, \ldots, q \right\}.
\end{align}
{\jb Given a family of sum-of-squares (SOS) polynomials $p_0, p_1, \ldots, p_q$, hence nonnegative}, it is direct to check that the following {\jb polynomial $P$ {\it in Putinar form}}
\begin{align}\label{eq:Putinar}
	P = p_0  + \sum_{i=1}^q p_i g_i,
\end{align}
{\jb is nonnegative on $G$}. Checking {\jb whether a given polynomial is a SOS reduces to solving a semidefinite program} and is thus amenable to {\jb efficient practical computation} \cite{vandenberghe1996semidefinite}. {\jb Actually dedicated software tools exist}
\cite{lofberg2009pre}. {\jb Hence under Assumption \ref{ass:algebra}, if $\valFun$ is a polynomial then Condition \ref{eq:globalLower} can be enforced 
by semidefinite constraints.} This is of course an approximation and in fact {\jb the SOS constraints (\ref{eq:setG})} are stronger than the {\jb original} positivity constraints, {\jb whence} the name {\jb ``reinforcement"}. But a counterpart of this approximation is that it is amenable to practical computation on moderate size problems which is not the case for general positivity constraints.

{\jb Conic convex} duality also holds for {\jb semidefinite programs}. In the {\jb present} context of control, the {\jb dual variables associated with the} SOS reinforcement of condition (\ref{eq:globalLower}) {\jb are ``moments" of the occupation measures} discussed 
in Section \ref{sec:conicDuality}. Hence the SOS approximation actually bears the name \textit{moment-SOS} approximation, see \cite{lasserre2010} for a comprehensive treatment.

\subsection{Convergence}
The positivity certificate in equation (\ref{eq:Putinar}) {\jb describes} a family of nonnegative polynomials over {\jb the set $G$} involving a family of SOS polynomials {\jb $\left\{p_i  \right\}_{i=0}^q$}. By increasing the degree {\jb allowed for these SOS polynomials $p_i$, one} provides a hierarchy of increasing families of {\jb polynomials nonnegative on $G$}. A relevant {\jb issue is:

	{\it What happens when we let the degrees of the SOS polynomials $\left\{p_i  \right\}_{i=0}^q$ defining this hierarchy goes to infinity?}} 

This {\jb issue is related to the question of the representation of nonnegative polynomials on
compact basic semi-algebraic sets}. {\jb Fortunately, powerful results from real algebraic geometry state that it is enough to work with certificates of the form of (\ref{eq:Putinar}) \cite{putinar1993positive}}. This usually translates in global convergence results: replacing nonnegativity constraint in condition (\ref{eq:globalLower}) by their SOS reinforcement and letting the degree of the SOS polynomials go to infinity is, in some sense, equivalent to the intractable constraints in condition (\ref{eq:globalLower}). {\jb Applications} of sufficient conditions to represent positive polynomials {\jb date} back to \cite{lasserre2001global} {\jb in static optimization and to \cite{lasserre2008nonlinear} in optimal control; see also \cite{lasserre2010,chesi2010lmi} for a more recent overview}. This methodology can be used for all the control problems described in Section \ref{sec:applications} to provide converging hierarchies of semidefinite approximations \cite{lasserre2008nonlinear,henrion2014convex,pauwels2014linear}, see also \cite{henrion2014optim} for a detailed overview.

\section{Numerical illustration}
\begin{figure}[t]
	\begin{picture}(140,100)

		\put(0,-45){\includegraphics[width=90mm]{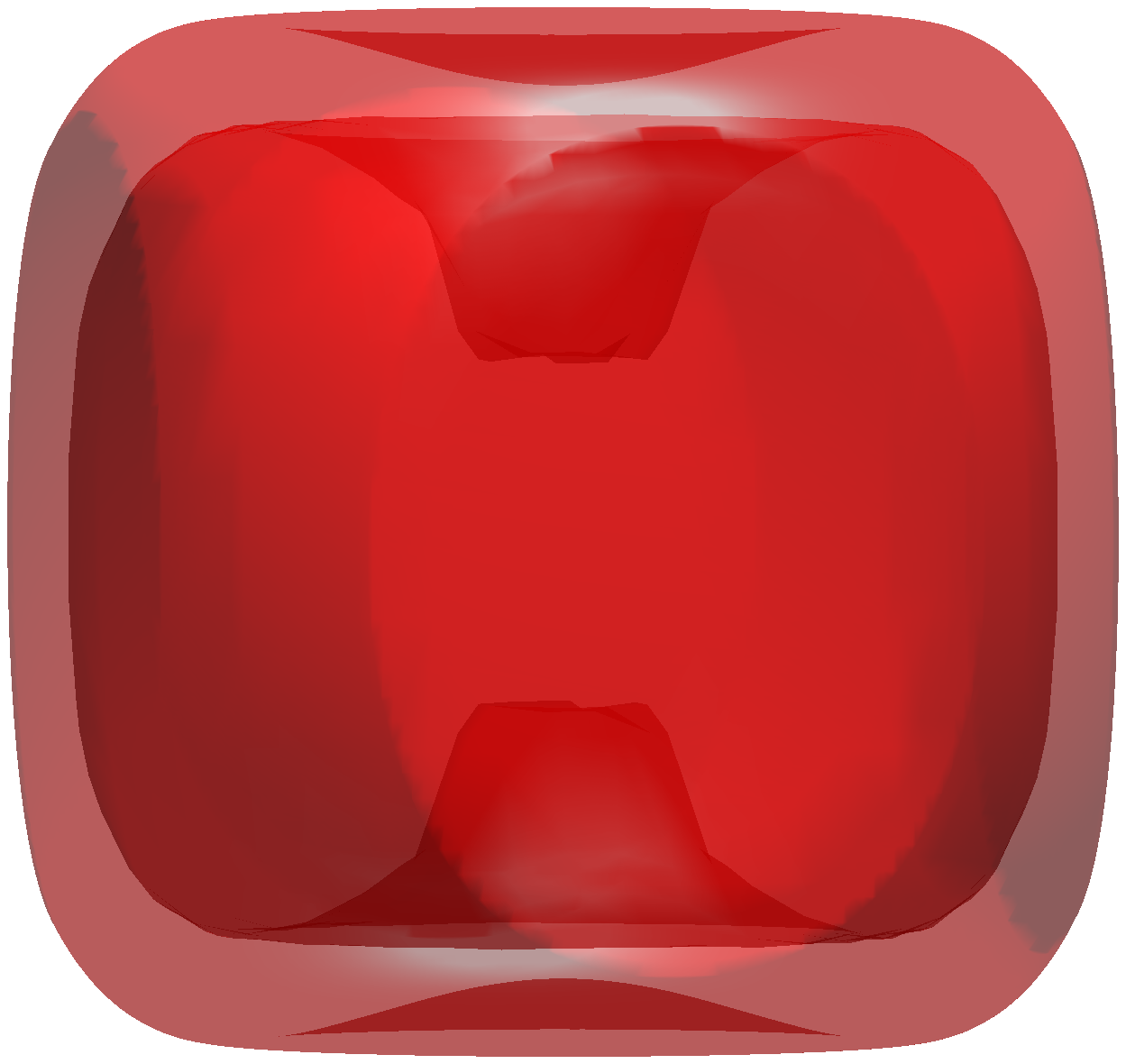}} 
		\put(210,-45){\includegraphics[width=90mm]{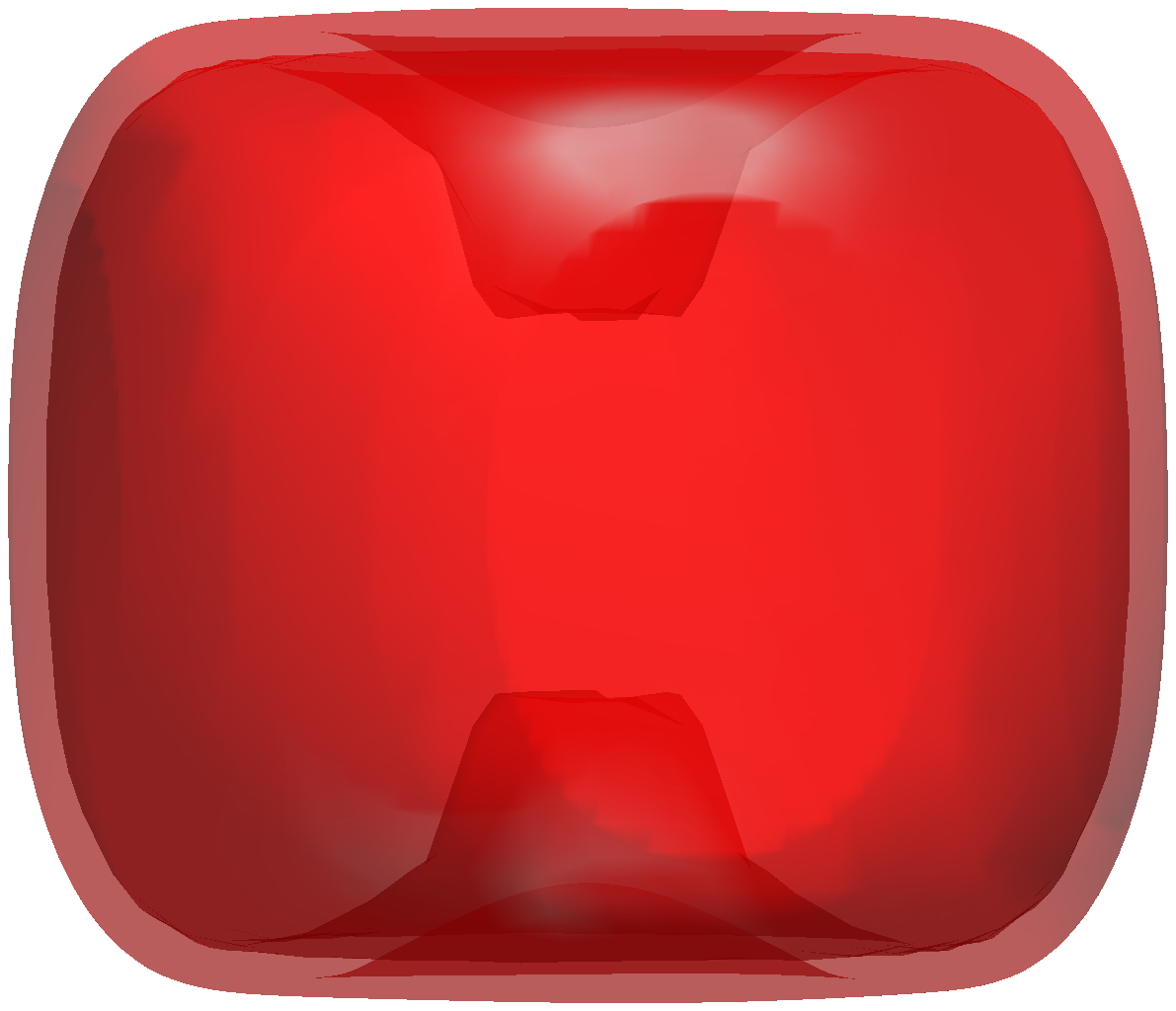}}

		\put(123,120){\footnotesize $d = 6$}
		\put(330,120){\footnotesize $d=10$}

	\end{picture}
	\caption{Sublevel set outer approximations (light red, larger) to the region of attraction $X_0$ (dark red, smaller) for different degrees of the SOS reinforcement. This was originally published in \cite{henrion2014convex}.}
	\label{fig:Brockett}
\end{figure}
In this section we briefly describe numerical results obtained when applying the SOS reinforcement techniques of Section \ref{sec:sosReinforcement} to the abstract optimization problems of Section \ref{sec:applications}. We choose a simple but non trivial nonlinear system: Brockett's integrator. This system is of importance in humanoid robotics since it is equivalent to the nonholonomic Dubbins vehicle \cite{soueres1996shortest}, a model of human walking \cite{arechavaleta2008optimality}, up to a change of variable \cite{devon2007kinematic}. The numerical results presented here relate to free terminal time optimal control problem, which is different from the fixed terminal time setting considered in this work. They still illustrate most important aspects of these simulations. Indeed, most of the ideas presented in Sections \ref{sec:bounds} and \ref{sec:applications} have direct equivalent in the free terminal time setting. In a nutshell, the terminal time in (\ref{eq:optimalControl}) is not fixed to be $1$ but is a decision variable of the problem. In this case, the value function $v^*$ as well as its lower approximations $v$ can be chosen to be independent of time. The numerical examples of this section were originally presented in \cite{henrion2014convex,lasserre2008nonlinear,pauwels2014inverse,pauwels2014linear}. All these examples were computed by combining the abstract infinite dimensional optimization problems of Section \ref{sec:applications} with the SOS reinforcement techniques of Section \ref{sec:sosReinforcement}.
\subsection{Brockett's integrator}
Brockett's integrator is a 3-dimensional nonlinear system with two dimensional control. We set $X = \left\{ x \in \RR^3: \|x\|_\infty \leq1 \right\}$, $U = \left\{ u \in \RR^2: \|u\|_2 \leq 1 \right\}$ and we let $X_T$ be the origin in $\RR^3$. The dynamics of the system are given by
\begin{align}
	\label{eq:brockett}
	f(x, u) = \left(
	\begin{array}{c}
		u_1\\
		u_2\\
		u_1 x_2 - u_2 x_1
	\end{array}
	\right),
\end{align}
where the subscripts denote the corresponding coordinates. All the following examples are related to the minimum time to reach the origin under the previous dynamical constraints. The value function of this problem is known and described in \cite[Theorem 1.36 and 1.41]{beals2000hamilton} and the corresponding optimal control is computed in \cite[Corollary 1]{prieur2005robust}. In what follows, $T(x)$ denotes the optimal time to reach the origin, starting at initial state $x$ under the dynamical constraints (\ref{eq:brockett}).

\begin{table}[t]
	\caption{Brockett's integrator, comparison of exact optimal time and SOS reinforcement. This was originally published in \cite{lasserre2008nonlinear}.}
	\label{table_example3}
	\begin{center}
		\begin{tabular}{||c|c|c||c|c|c||}
			\hline
			\multicolumn{3}{||c||}{SOS reinforcement.}&\multicolumn{3}{|c|}{Optimal time.}\\
			\hline
			1.7979&2.3614&3.2004&1.8257 & 2.3636&3.2091\\
			\hline
		 	2.3691 & 2.6780&3.3341&2.5231 & 2.6856 & 3.3426\\
		  \hline
			2.8875&3.0654 & 3.5337&3.1895&3.1008 & 3.5456 \\
			\hline
		\end{tabular}
	\end{center}
\end{table}
\subsection{Region of attraction}
For this application, the final time is set to $1$ and initial time is set to $0$. The region of attraction described in Section \ref{sec:applications} is the set $X_0$ of initial states for which there exists a feasible trajectory reaching the origin in time less or equal to $1$. In other words, it consists of the set of initial states $x$, for which $T(x) \leq 1$. This quantity is computable explicitly \cite{beals2000hamilton}. Combining the formulation in equation (\ref{eq:ROAopt}) with the SOS reinforcement technique described in Section \ref{sec:sosReinforcement}, we get sublevel sets which are outer approximations of $X_0$. This is represented in Figure \ref{fig:Brockett} which compares the true region of attraction and its outer approximation in $\RR^3$.

\subsection{Minimum time direct optimal control}
For the direct optimal control problem, we are interested in the value of the optimal time $T(x)$. Following \cite{lasserre2008nonlinear}, we combine the formulation in (\ref{eq:ROCP}) with the SOS reinforcement technique described in Section \ref{sec:sosReinforcement}. As a result, we get a lower approximation of the optimal time $T(x)$. This is illustrated in Table \ref{table_example3} for different initial conditions: $x_1 = 1$ and $x_i \in \left\{ 1,2,3 \right\}$, $i=2,3$. As expected, we obtain lower bounds on the optimal time which is known analytically. For these examples, the approximation is reasonably accurate.

\subsection{Inverse optimal control}
\begin{figure}[t]
	\centering
	\includegraphics[width=0.6\textwidth]{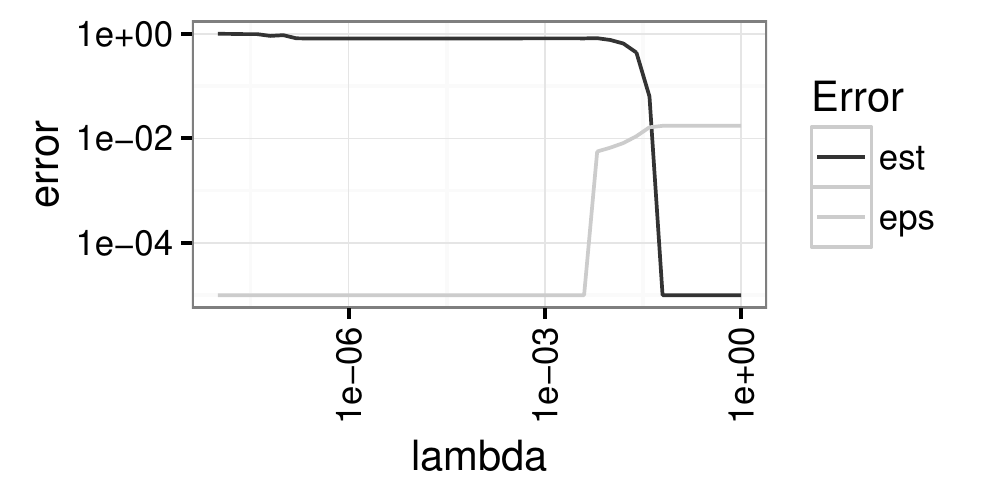}
	\caption{Error versus a regularization parameter $\lambda$ for the inverse optimal control problem. Estimation error (est) is a measure of distance between a candidate solution computed with (\ref{eq:IOCP}) and the constant Lagrangian (normalized between $0$ and $1$). Epsilon error (eps) corresponds to the value of $\epsilon$ in (\ref{eq:IOCP}). The Lagrangian is looked for among all 3-variate polynomials of degree up to $4$. These results were originally presented in \cite{pauwels2014linear}.}
	\label{fig:deterministic}
\end{figure}

In this example, we are interested in recovering the minimum time Lagrangian (constant) from optimal trajectories which reach the origin in minimal time under dynamical constraints (\ref{eq:brockett}). These trajectories can be computed analytically \cite{prieur2005robust}. As outlined in Figure \ref{fig:directIllustr}, the trajectories constitute an input of the inverse problem and the output is a Lagrangian function. In order to find this function, we follow the work of \cite{pauwels2014inverse,pauwels2014linear} which combines the abstract problem in (\ref{eq:IOCP}) with SOS reinforcement techniques described in Section \ref{sec:sosReinforcement} and additional constraints. We emphasize that the problem of Lagrangian identification is much less well posed than the direct optimal control problem and that accuracy of solutions highly depend on prior information about expected Lagrangians. In \cite{pauwels2014inverse,pauwels2014linear}, it is shown that the success of such a procedure requires careful normalization and prior knowledge enforcement (sparsity through a regularization term). We do not describe the details of the procedure here and refer to \cite{pauwels2014inverse,pauwels2014linear} for more details. This formulation includes a regularization parameter denoted by $\lambda$. Figure \ref{fig:deterministic} presents measures of inverse optimality accuracy, (the value of $\epsilon$ in (\ref{eq:IOCP})), and estimation accuracy (a distance to the constant Lagrangian, the true Lagrangian of minimum time optimal control), for various values of this parameters. The input is made of optimal time trajectories and Figure \ref{fig:deterministic} illustrates that the original Lagrangian can be recovered with a reasonable inverse optimality accuracy for some values of the regularization parameter close to 1.

\section*{Acknowledgements}
This work was partly supported by project 16-19526S of the {\it Grant Agency of the Czech Republic}, project ERC-ADG TAMING 666981,ERC-Advanced Grant of the {\it European Research Council} and grant number FA9550-15-1-0500 from the {\it Air Force Office of Scientific Research, Air Force Material Command}.

\end{document}